\documentclass[12pt,a4paper]{article}
\usepackage{amsmath, amsfonts, amsthm,amssymb, enumerate}
\theoremstyle{plain}
\newtheorem{thm}{Theorem}[section]
\newtheorem{lem}[thm]{Lemma}
\newtheorem{cor}[thm]{Corollary}
\newtheorem{prop}[thm]{Proposition}

\theoremstyle{definition}
\newtheorem{defi}[thm]{Definition}

\title{\bf{On linear groups over weakly locally finite division rings\footnote{The first and second authors were supported by a grant NAFOSTED (Vietnam). The third author was supported by Vietnam National University (VNU-HCMC) under grant number B2012-18-31} }}

\author{Bui Xuan Hai\footnote{Corresponding author, Faculty of Mathematics and Computer Science, University of Science, VNU-HCMC, 227 Nguyen Van Cu Str., Dist. 5, HCM-City, Vietnam,  e-mail: bxhai@hcmus.edu.vn},
Mai Hoang Bien\footnote{Department of Basic Sciences, University of Architecture, 196 Pasteur Str., Dist. 1, HCM-City, Vietnam, e-mail:  maihoangbien012@yahoo.com}, and
Trinh Thanh Deo\footnote{Faculty of Mathematics and Computer Science, University of Science, VNU-HCMC, 227 Nguyen Van Cu Str., Dist. 5, HCM-City, Vietnam,  e-mail: ttdeo@hcmus.edu.vn}} 

\begin{document}
\baselineskip=18pt
\maketitle
\def\Q{\mathbb{Q}}\def\F{\mathbb{F}}
\newcommand{\hpt}[2]{\left\{\begin{array}{#1} #2\end{array}\right.}
\begin{abstract}
In this paper,  we give the definition of {\em weakly locally finite} division rings and we show that the class of these rings  strictly contains the class of locally finite division rings.  Further, we study multiplicative subgroups  in these rings. Some skew linear groups are also considered. Our  new obtained results generalize previous results for centrally finite case.
\end{abstract}



\newpage
\section{Introduction}

Let $D$ be a division ring with center  $F$. Recall that $D$ is {\em centrally finite} if $D$ is a finite dimensional vector space over $F$; $D$ is {\em locally  finite} if for every finite subset $S$ of $D$, the division subring $F(S)$  of $D$ generated by $S\cup F$ is a finite dimensional vector space over $F$. If $a$ is an element from $D$, then we have the field extension $F\subseteq F(a)$. Obviously, $a$ is {\em algebraic} over $F$ if and only if this extension is finite. We say that a non-empty subset $S$ of $D$ is {\em algebraic} over $F$ if every element of $S$ is algebraic over $F$. A division ring $D$ is  {\em algebraic} over center $F$ (briefly, $D$ is {\em algebraic}), if every element of $D$ is algebraic over $F$. Clearly, a locally  finite division ring is algebraic. It was conjectured that any algebraic division ring is locally  finite (this is known as the Kurosch's Problem for division rings \cite{Kha}). However, this problem remains still open in general. There exist locally  finite division rings which are not centrally finite. In this paper, we define a {\em weakly locally  finite} division ring as a division ring in which for every finite subset $S$, the division subring generated by $S$ is centrally finite. It can be shown that every locally finite division ring is weakly locally finite. The converse is not true. Moreover, in the text, we give one example of weakly locally finite division ring which is not even algebraic. Next, we study subgroups in weakly locally  finite division rings. In particular, we give the affirmative  answer to one of Herstein's conjectures \cite{her} for these rings. Some linear groups  are also investigated.  Our new obtained results generalize previous results for centrally finite case. The symbols and notation we use in this paper are standard and they should be found in the literature on subgroups in division rings and on skew linear groups. In particular, for a division ring $D$ we denote by $D^*$ and $D'$ the multiplicative group and the derived group of $D$ respectively.

 \section{Definitions and examples}

\begin{defi}\label{def:2.1} 
We say that a division ring $D$  is  {\em weakly locally  finite } if  for every finite subset $S$ of $D$, the division subring  generated by $S$ in $D$ is  centrally finite.
\end{defi}
It can be shown that every division subring of a centrally finite division ring is itself centrally finite. Using this fact, it is easy to see that every locally finite division ring is weakly locally finite. 

%----------------------------------------
Our purpose in this section is to construct an example showing the difference between  the class of locally finite division rings and the class of weakly locally finite division rings. In order to do so, following the general Mal'cev-Neumann construction of Laurent series rings, we construct a Laurent series ring with a base ring which is an extension of the field $\mathbb{Q}$ of rational numbers. The ring we construct in the following proposition is weakly locally finite but it is not even algebraic. 

\begin{prop}
There exists a weakly locally finite division ring which is not algebraic.
\end{prop}

\begin{proof}
Denote by $G=\bigoplus\limits_{i=1}^{\infty}\mathbb{Z}$ the  direct sum of infinitely many copies of the additive group $\mathbb{Z}$. For any positive integer $i$, denote by $x_i= (0, \ldots, 0, 1, 0, \ldots)$ the element of $G$ with  $1$ in the $i$-th position and $0$ elsewhere.  Then $G$ is a free abelian group generated by all $x_i$ and every element $x\in G$ is written uniquely in the form
$$x=\sum\limits_{i\in I}n_i x_i,\eqno(1)$$
with $n_i\in\mathbb{Z}$ and some finite set $I$.

Now, we define an order in $G$ as follows:

For elements  $x=(n_1, n_2, n_3, \ldots)$ and  $y=(m_1, m_2, m_3, \ldots)$ in $G$, define  $x<y$ if either  $n_1<m_1$  or there exists $k\in \mathbb{N}$ such that  $n_1=m_1,\ldots, n_k=m_k$ and  $n_{k+1}<m_{k+1}$. Clearly, with this order $G$ is a totally ordered set.

Suppose that  $p_1<p_2<\ldots <p_n<\ldots$ is a sequence of prime numbers and 
$K=\mathbb{Q}(\sqrt{p_1},\sqrt{p_2},\ldots)$ is the subfield of the field $\mathbb{R}$ of real numbers generated by $\mathbb{Q}$ and  $\sqrt{p_1},\sqrt{p_2},\ldots$, where $\mathbb{Q}$ is the field of rational numbers. For any $i\in\mathbb{N}$, suppose that  $f_i:K\to K$ is  $\mathbb{Q}$-isomorphism satisfying the following condition: 
$$f_i(\sqrt{p_i})=-\sqrt{p_i};\quad \text{and } \ f_i(\sqrt{p_j})=\sqrt{p_j}\quad \text{for any }j\neq i.$$ 
It is easy to verify that $f_i f_j=f_j f_i$ for any $i,j\in\mathbb{N}.$ \\[-6pt]

\noindent {\em $\bullet$ Step 1. Proving that, for  $x\in K, f_i(x)=x$ for any $i\in \mathbb{N}$ if and only if  $x\in\mathbb{Q}$:}

The converse is obvious. Now, suppose that $x\in K$ such that $f_i(x)=x$ for any $i\in \mathbb{N}$. By setting $K_0=\mathbb{Q}$ and $K_i=\mathbb{Q}(\sqrt{p_1}, \ldots, \sqrt{p_i})$ for $i\geq 1$, we have the following ascending series:
$$K_0\subset K_1\subset\ldots\subset K_i\subset\ldots$$
If $x\not\in \mathbb{Q}$, then there exists $i\geq 1$ such that $x\in K_i\setminus K_{i-1}$. So, we have
$x=a+b\sqrt{p_i}$, with $a, b\in K_{i-1}$ and $b\neq 0.$
Since $f_i(x)=x,0=x-f_i(x)=2b\sqrt{p_i}$, a contradiction.\\[-6pt] 

\noindent {\em $\bullet$  Step 2. Constructing a Laurent series ring:} 

For any  $x=(n_1, n_2, ...)=\sum\limits_{i\in I} n_i x_i\in G$, define $\Phi_x:=\prod\limits_{i\in I} f_i^{n_i}.$ Clearly $\Phi_x\in Gal(K/\mathbb{Q})$ and the map $\Phi: G\rightarrow Gal(K/\mathbb{Q}),$
defined by $\Phi(x)=\Phi_x$ is a group homomorphism. 
The following conditions hold.

\begin{enumerate}[{\rm i)}]
  \item $\Phi(x_i)=f_i$ for any $i\in \mathbb{N}$. 
  \item If $x=(n_1, n_2, \ldots)\in G$, then $\Phi_x(\sqrt{p_i})=(-1)^{n_i} \sqrt{p_i}$.
\end{enumerate}

For the convenience, from now on we write the operation in  $G$ multiplicatively. For $G$ and $K$ as above, consider formal sums of the form
$$\alpha=\sum\limits_{x\in G} a_x x,\quad a_x\in K.$$
For  such an $\alpha$, define the support of $\alpha$ by $supp(\alpha)=\{x\in G: a_x\neq 0\}$. Put
 $$D=K((G,\Phi)):=\Big\{\alpha=\sum\limits_{x\in G} a_x x, a_x\in K : supp(\alpha) \text{ is  well-ordered }\Big\}.$$ 
For  $\alpha=\sum\limits_{x\in G} a_x x$ and $\beta=\sum\limits_{x\in G} b_x x$ from $D$, define 
$$\alpha+\beta=\sum\limits_{x\in G} (a_x+b_x) x;\quad \text{and}\quad 
\alpha\beta=\sum\limits_{z\in G} \Big(\sum\limits_{xy=z}a_x \Phi_x(b_y)\Big) z.$$
With  operations defined as above, $D=K((G,\Phi))$ is a division ring (we refer to [6, pp. 243-244]).
Moreover, the following conditions hold.

\begin{enumerate}
  \item[iii)]  For any $x \in G, a\in K$, $xa=\Phi_x(a) x$.
  \item[iv)] For any $i\neq j$, $x_i\sqrt{p_i}=-\sqrt{p_i}x_i$ and  $x_j\sqrt{p_i}=\sqrt{p_i}x_j$.
  \item[v)] For any $i\neq j$ and $n\in \mathbb{N}$,  $x_i^n \sqrt{p_i}=(-1)^n\sqrt{p_i} x_i^n$ and  $x_j^n \sqrt{p_i}=\sqrt{p_i}x_j^n$.
\end{enumerate}

\noindent {\em $\bullet$ Step 3. Finding the center of $D$:}

Put $H=\{x^2: x\in G\}$ and $\mathbb{Q}((H))=\Big\{\alpha=\sum\limits_{x\in H} a_x x, a_x\in \mathbb{Q}: supp(\alpha) \text{ is well-ordered }\Big\}.$
It is easy to check that  $H$ is a subgroup of $G$ and for every $x\in H$, $\Phi_x=Id_K$.

Denote by $F$ the center of $D$. We claim that $F=\mathbb{Q}((H)).$ Suppose that $\alpha=\sum\limits_{x\in H} a_x x \in \mathbb{Q}((H))$. Then, for every  $\beta=\sum\limits_{y\in G} b_y y \in D$, we have 
$\Phi_x (b_y)= b_y$ and $\Phi_y (a_x)= a_x$. Hence
\begin{eqnarray*}
\alpha\beta&=&\sum\limits_{z\in G} \Big(\sum\limits_{xy=z}a_x \Phi_x(b_y)\Big) z
=\sum\limits_{z\in G} \Big(\sum\limits_{xy=z}a_x b_y\Big) z,\\
\beta\alpha&=&\sum\limits_{z\in G} \Big(\sum\limits_{xy=z}b_y \Phi_y(a_x)\Big) z
 =\sum\limits_{z\in G} \Big(\sum\limits_{xy=z}a_x b_y\Big) z.
\end{eqnarray*} 
Thus, $\alpha\beta=\beta\alpha$ for every $\beta\in D$, so $\alpha \in F$.

Conversely, suppose that  $\alpha=\sum\limits_{x\in G} a_x x \in F.$
Denote by  $S$ the set of all elements $x$ appeared in the expression of $\alpha$. Then, it suffices to prove that  $x\in H$ and $a_x\in \mathbb{Q}$ for any $x\in S$. 
In fact, since $\alpha \in F$, we have $\sqrt{p_i}\alpha=\alpha \sqrt{p_i}$ and $\alpha x_i =x_i\alpha$ for any $i\geq 1$,
i.e.
$\sum\limits_{x\in S} \sqrt{p_i}a_x x=\sum\limits_{x\in S} \Phi_x(\sqrt{p_i})a_x x$ and $\sum\limits_{x\in S} a_x (x x_i)=\sum\limits_{x\in S} \Phi_{x_i}(a_x) (x_i x).$
Therefore, by conditions mentioned in the beginning of {\em  Step 2}, we have $\sqrt{p_i}a_x =\Phi_x(\sqrt{p_i})a_x=(-1)^{n_i}\sqrt{p_i}a_x$ and $a_x=\Phi_{x_i}(a_x)=f_i(a_x)$ for any $x=(n_1, n_2, \ldots)\in S$.
From the first equality it follows that  $n_i$ is even for any  $i\geq 1$. Therefore $x\in H$. From the second equality it follows that $a_x=f_i (a_x)$ for any $i\geq 1$. So by {\em Step 1}, we have $a_x\in \mathbb{Q}$. Therefore $\alpha \in \mathbb{Q}((H)).$
Thus,  $F=\mathbb{Q}((H)).$\\[-6pt]

\noindent {\em $\bullet$ Step 4. Proving that $D$ is not algebraic over $F$:}

Suppose that  $\gamma = x_1^{-1}+x_2^{-1}+\ldots$ is an infinite formal sum.  
Since  $x_1^{-1}<x_2^{-1}<\ldots,  supp(\gamma)$ is  well-ordered. Hence $\gamma\in D$.
Consider the equality
$$a_0+a_1\gamma+a_2\gamma^2+\ldots+a_n\gamma^n=0,\quad a_i\in F.\eqno(2)$$
Note that $X=x_1^{-1}x_2^{-1}...x_n^{-1}$ does not appear in the expressions of $\gamma, \gamma^2,\ldots,\gamma^{n-1}$ and the coefficient of  $X$ in the expression of $\gamma^n$ is  $n!$. Therefore, the coefficient of  $X$ in the expression on left side of the equality $(2)$ is $a_n.n!$. It follows that  $a_n=0$. 
By induction, it is easy to see that  $a_0=a_1=\ldots =a_n=0$.
Hence, for any  $n\in\mathbb{N}$, the set $\{1,\gamma, \gamma^2, \ldots,\gamma^n\}$ is  independent over $F$. Consequently,  $\gamma$ is not algebraic over $F$.\\[-6pt]

\noindent {\em $\bullet$  Step 5. Constructing a division subring of $D$ which is a weakly locally finite:}

Consider the element $\gamma$ from {\em Step 4}. For any $n\geq 1$, put 
$$R_n=F(\sqrt{p_1},\sqrt{p_2},\ldots, \sqrt{p_n},x_1,x_2,\ldots, x_n, \gamma),$$
and  $R_\infty=\bigcup\limits_{n=1}^\infty R_n.$
First, we prove that $R_n$ is centrally finite for each positive integer $n$.
Consider the element  $$\gamma_n=x_{n+1}^{-1}+x_{n+2}^{-1}+\ldots \quad \text{(infinite formal sum)}.$$ 
Since $\gamma_n =\gamma -(x_1^{-1}+x_2^{-1}+\ldots+x_n^{-1})$, we conclude that 
$\gamma_n\in R_n$ and 
$$F(\sqrt{p_1},\sqrt{p_2},\ldots, \sqrt{p_n},x_1,x_2,\ldots , x_n, \gamma)=F(\sqrt{p_1},\sqrt{p_2},\ldots, \sqrt{p_n},x_1,x_2,\ldots, x_n, \gamma_n).$$
Note that $\gamma_n$ commutes with all  $\sqrt{p_i}$ and all $x_i$ (for $i=1,2,...,n$).  Therefore
\begin{eqnarray*}
R_n&=&F(\sqrt{p_1},\sqrt{p_2},\ldots, \sqrt{p_n},x_1,x_2,\ldots,x_n, \gamma_n)\\
&=&F(\gamma_n)(\sqrt{p_1},\sqrt{p_2},\ldots, \sqrt{p_n},x_1,x_2,\ldots, x_n).
\end{eqnarray*}
In combination with the equalities  $(\sqrt{p_i})^2=p_i, x_i^2\in F$, $\sqrt{p_i}x_j=x_j\sqrt{p_i}, i\ne j,$ $\sqrt{p_i}x_i=-x_i\sqrt{p_i},$
it follows that every element $\beta$ from  $R_n$ can be written in the form 
$$\beta = \sum\limits_{0\leq \varepsilon_i, \mu_i\leq 1} a_{(\varepsilon_1, ..., \varepsilon_n,  \mu_1, ..., \mu_n)} (\sqrt{p_1})^{\varepsilon_1}\ldots (\sqrt{p_n})^{\varepsilon_n} x_1^{\mu_1} \ldots x_n^{\mu_n}, \text{ where }a_{(\varepsilon_1, ..., \varepsilon_n\mu_1, ..., \mu_n)}\in F(\gamma_n).$$
Hence  $R_n$ is a vector space over  $F(\gamma_n)$ having the finite set $B_n$ which consists of the products  
$$(\sqrt{p_1})^{\varepsilon_1}\ldots (\sqrt{p_n})^{\varepsilon_n} x_1^{\mu_1} \ldots x_n^{\mu_n}, 0\le \varepsilon_i, \mu_i\le 1$$
 as a base.
Thus,  $R_n$ is a finite dimensional vector space over $F(\gamma_n)$. 
Since  $\gamma_n$ commutes with all  $\sqrt{p_i}$ and all $x_i, F(\gamma_n)\subseteq Z(R_n)$. It follows that  $\dim_{Z(R_n)}R_n\le \dim_{F(\gamma_n)}R_n<\infty$ and consequently, $R_n $ is centrally finite. 

For any finite subset $S\subseteq R_\infty$, there exists  $n$ such that $S\subseteq R_n$. Therefore, the division subring of $R_\infty$, generated by $S$ over $F$ is contained in $R_n$, which is centrally finite. Thus,  $R_\infty$ is weakly locally  finite.\\[-6pt] 

\noindent {\em $\bullet$  Step 6. Finding the center of $R_\infty$:}

We claim that $Z(R_\infty)=F$. Put $S_n=\{\sqrt{p_1}, ..., \sqrt{p_n}, x_1, ..., x_n\}.$
Since for any  $i\neq j$,  
$x_i^2 , (\sqrt{p_i})^2 \in F$, $x_i x_j = x_j x_i$, $\sqrt{p_i}\sqrt{p_j}=\sqrt{p_j}\sqrt{p_i}$, $x_i\sqrt{p_j}=\sqrt{p_j}x_i$, $x_i\sqrt{p_i}=-\sqrt{p_i}x_i$, 
every element from  $F[S_n]$ can be expressed in the form  
$$\alpha = \sum_{0\leq \varepsilon_i, \mu_i\leq 1} a_{(\varepsilon_1, ..., \varepsilon_n,  \mu_1, ..., \mu_n)} (\sqrt{p_1})^{\varepsilon_1}\ldots (\sqrt{p_n})^{\varepsilon_n} x_1^{\mu_1} \ldots x_n^{\mu_n},\quad  a_{(\varepsilon_1, ..., \varepsilon_n,  \mu_1, ..., \mu_n)}\in F.\eqno(3)$$
Moreover, the set  $\mathcal{B}_n$ consists of products  $(\sqrt{p_1})^{\varepsilon_1}\ldots (\sqrt{p_n})^{\varepsilon_n} x_1^{\mu_1} \ldots x_n^{\mu_n}, 0\leq \varepsilon_i, \mu_i\leq 1$ is finite of  $2^{2n}$ elements. 
Hence, $F[S_n]$ is a finite dimensional vector space over $F$. So, it follows that $F[S_n]=F(S_n)$. Therefore, every element from  $F(S_n)$ can be expressed in the form $(3)$. 

In the first, we show that $Z(F(S_1))=F$. Thus, suppose that $\alpha\in Z(F(S_1))$. 
Since $x_1^2, (\sqrt{p_1})^2=p_1\in F$  and  $x_1 \sqrt{p_1}=-\sqrt{p_1}x_1$, every element  $\alpha \in F(S_1) = F(\sqrt{p_1}, x_1)$ can be expressed in the following form:
$$\alpha = a+b\sqrt{p_1}+ c x_1+d \sqrt{p_1} x_1, \quad a,b,c,d\in F.$$
Since $\alpha$ commutes with $x_1$ and $\sqrt{p_1}$, we have  
$$ax_1+b\sqrt{p_1}x_1+ c x_1^2+d \sqrt{p_1} x_1^2=ax_1-b\sqrt{p_1}x_1+ c x_1^2-d \sqrt{p_1} x_1^2,$$
and
$$a\sqrt{p_1}+bp_1- c\sqrt{p_1} x_1-dp_1 x_1= a\sqrt{p_1}+bp_1+ c\sqrt{p_1} x_1+d p_1 x_1.$$
From the first equality it follows that $b=d=0$, while from the second equality we obtain  $c=0$. Hence, $\alpha = a \in F$ and consequently, $Z(F(S_1))=F$. 

Suppose that $n\geq 1$ and $\alpha\in Z(F(S_n))$. By (3), $\alpha$ can be expressed in the form
$$\alpha=a_1+a_2\sqrt{p_n}+a_3x_n+a_4\sqrt{p_n}x_n, \mbox{ with } a_1, a_2, a_3, a_4\in F(S_{n-1}).$$

From the equality  $\alpha x_n=x_n\alpha$, it follows that
$$a_1x_n+a_2\sqrt{p_n}x_n+a_3x_n^2+a_4\sqrt{p_n}x_n^2=a_1x_n-a_2\sqrt{p_n}x_n+a_3x_n^2-a_4\sqrt{p_n}x_n^2.$$

Therefore, $a_2+a_4x_n=0$ and consequently we have $a_2=a_4=0$. Now, from the equality $\alpha\sqrt{p_n}=\sqrt{p_n}\alpha$, we have $a_1\sqrt{p_n}-a_3\sqrt{p_n}x_n=a_1\sqrt{p_n}+a_3\sqrt{p_n}x_n$ and it follows that $a_3=0$. Therefore, $\alpha=a_1\in F(S_{n-1})$ and this means that $\alpha\in Z(F(S_{n-1}))$. Thus, we have proved that $Z(F(S_n))\subseteq Z(F(S_{n-1}))$. By induction we can conclude that $Z(F(S_n))\subseteq Z(F(S_1))$ for any positive integer $n$. Since $F\subseteq Z(F(S_n))\subseteq Z(F(S_1))=F$, it follows that $Z(F(S_n))=F$ for any positive integer $n$. Now, suppose that $\alpha\in Z(R_{\infty})$. Then, there exists some $n$ such that $\alpha\in R_n$ and clearly $\alpha\in Z(F(S_n))=F$. Hence $Z(R_{\infty})=F$. \\[-6pt]

\noindent {\em $\bullet$ Step 7. Proving that $R_\infty$ is not algebraic over $F$:}

It was shown in  {\em Step 4}  that  $\gamma \in R_\infty$ is not algebraic over $F$.
\end{proof}

\section{Herstein's Conjecture for weakly locally finite \\division rings}

Let $K \varsubsetneq D$ be  division rings . Recall that an element $x\in D$ is {\em radical} over $K$ if there exists some positive integer $n(x)$ depending on $x$ such that $x^{n(x)}\in K$. A subset $S$ of $D$ is {\em radical} over $K$ if every element from $S$ is radical over $K$. In 1978, I.N. Herstein (cf. \cite{her}) conjectured that given a subnormal subgroup $N$ of $D^*$, if $N$ is radical over center $F$ of $D$, then $N$ is central, i. e. $N$ is contained in $F$. Herstein, himself in the cited above paper proved this fact for the special case, when $N$ is torsion group. However, the problem remains  still open in general. In \cite{hai-huynh}, it was proved that  this  conjecture is  true in the finite dimensional case. In this section, we  shall prove that this conjecture  is also true for weakly locally finite division rings. First, we note the following two lemmas  we need for our further purpose.

\begin{lem}\label{lem:3.1} 
Let $D$ be a division ring with center $F$. If $N$ is a subnormal subgroup of $D^*$, then $Z(N)=N\cap F$.
\end{lem}

\begin{proof} 
If $N$ is contained in $F$, then there is nothing to prove. Thus, suppose that $N$ is non-central. By [9, 14.4.2, p. 439], $C_D(N)=F$. Hence $Z(N)\subseteq N\cap F$. Since the inclusion $N\cap F\subseteq Z(N)$ is obvious, $Z(N)= N\cap F$.
\end{proof}

\begin{lem}\label{lem:3.2}
If  $D$ is a weakly locally  finite division ring, then $Z(D')$ is a torsion group.
\end{lem}
\begin{proof} By Lemma  \ref{lem:3.1}, $Z(D')=D'\cap F$. For any $x\in Z(D')$, there exists some positive integer $n$ and some $a_i, b_i\in D^*, 1\leq i\leq n$, such that 
$$x=a_1b_1a_1^{-1}b_1^{-1}a_2b_2a_2^{-1}b_2^{-1}\ldots a_nb_na_n^{-1}b_n^{-1}.$$
Set $S:=\{a_i,b_i: 1\leq i\leq n\}$.
Since $D$ is weakly locally finite, the division subring $L$ of $D$ generated by $S$ is centrally finite.
Put $n=[L:Z(L)].$ 
Since $x\in F$, $x$ commutes with every element of $S$. Therefore,  $x$ commutes with every element of  $L$, and consequently,  $x\in Z(L)$. So,  
$$x^n=N_{L/Z(L)}(x)=N_{L/Z(L)}(a_1b_1a_1^{-1}b_1^{-1}a_2b_2a_2^{-1}b_2^{-1}\ldots a_nb_na_n^{-1}b_n^{-1})=1.$$ 
Thus, $x$ is torsion.
\end{proof}
 
In [4, Theorem 1], Herstein proved that, if in a division ring $D$ every multiplicative commutator $aba^{-1}b^{-1}$ is torsion, then $D$ is commutative. Further, with the assumption that $D$ is a finite dimensional vector space over its center $F$, he proved [4, Theorem 2] that, if every multiplicative commutator  in $D$ is radical over $F$, then $D$ is commutative. Now, using Lemma \ref{lem:3.2}, we can carry over  the last fact for weakly locally finite division rings.

\begin{thm}\label{thm:3.2}
Let $D$ be a weakly locally  finite division ring with center $F$. If every multiplicative commutator in $D$ is radical over $F$, then  $D$ is commutative.
\end{thm}
\begin{proof} For any $a, b\in D^*$, there exists a positive integer $n=n_{ab}$ depending on $a$ and $b$ such that $(aba^{-1}b^{-1})^n\in F.$ Hence, by  Lemma  \ref{lem:3.2}, it follows that $aba^{-1}b^{-1}$ is torsion. Now, by [4, Theorem 1], $D$ is commutative.
\end{proof}

The following theorem gives the affirmative answer to Conjecture 3 in \cite{her} for  weakly locally finite division rings.

\begin{thm}\label{thm:3.3}
Let $D$ be a weakly locally  finite division ring with center $F$ and  $N$ be a  subnormal subgroup of  $D^*$. If  $N$ is radical over $F$, then $N$ is central, i.e. $N$ is contained in $F$.
\end{thm}
\begin{proof}
Consider the subgroup  $N'=[N,N]\subseteq D'$ and suppose that $x\in N'$. Since $N$ is radical over $F$, there exists some positive integer  $n$ such that $x^n\in F$. Hence $x^n \in F\cap D'=Z(D')$. By Lemma  \ref{lem:3.2},  $x^n$ is torsion, and consequently, $x$ is torsion too. Moreover, since  $N$ is subnormal in $D^*$, so is $N'$. Hence, by  [4, Theorem 8], $N'\subseteq F$. Thus, $N$ is solvable, and by [9, 14.4.4, p. 440], $N\subseteq F$.
\end{proof}

In Herstein's Conjecture a subgroup $N$ is required to be radical over center $F$ of $D$. What happen if $N$ is required to be radical over some proper division subring of $D$ (which not necessarily coincides with $F$)? In the other words, the following question should be interesting: {\em ``Let $D$ be a division ring and $K$ be a proper division subring of $D$ and given a subnormal subgroup $N$ of $D^*$. If $N$ is radical over $K$, then is it contained in center $F$ of $D$?"} In the following we give the affirmative answer to this question for a weakly locally finite ring $D$ and a normal subgroup $N$. 

\begin{lem}\label{lem:3.4} 
Let $D$ be a weakly locally  finite division ring with center $F$ and $N$ be a subnormal subgroup of $D^*$. If for every elements $x, y\in N$, there exists some positive integer $n_{xy}$ such that $x^{n_{xy}}y=yx^{n_{xy}}$, then $N\subseteq F$.
\end{lem}
\begin{proof} 
Since $N$ is subnormal in $D^*$, there exists the following series of subgroups
$$N=N_1\vartriangleleft N_2\vartriangleleft\ldots\vartriangleleft N_r=D^*.$$
Suppose that $x, y\in N$. Let $K$ be the division subring of $D$ generated by $x$ and $y$. 
Then, $K$ is centrally finite.
By putting $M_i=K\cap N_i, \forall i\in\{1, \ldots, r\}$ we obtain the following series of subgroups
$$M_1\vartriangleleft M_2\vartriangleleft\ldots\vartriangleleft M_r=K^*.$$
For any $a\in M_1\leq N_1=N$, suppose that $n_{ax}$ and $n_{ay}$ are positive integers such that
$a^{n_{ax}}x=xa^{n_{ax}}$ and $a^{n_{ay}}y=ya^{n_{ay}}.$
Then, for $n:=n_{ax}n_{ay}$ we have 
$a^n=$ $(a^{n_{ax}})^{n_{ay}}$ $=(xa^{n_{ax}}x^{-1})^{n_{ay}}$ $=xa^{n_{ax}n_{ay}}x^{-1}$ $=xa^nx^{-1},$
and
$a^n$ $=(a^{n_{ay}})^{n_{ax}}$ $=(ya^{n_{ay}}y^{-1})^{n_{ax}}$ $=ya^{n_{ay}n_{ay}}y^{-1}$ $=ya^ny^{-1}.$
Therefore $a^n\in Z(K)$. Hence $M_1$ is radical over $Z(K)$. By  Theorem \ref{thm:3.3}, $M_1\subseteq Z(K)$. In particular, $x$ and $y$  commute with each other. Consequently, $N$ is abelian group. By [9, 14.4.4, p. 440],  $N\subseteq F$.
\end{proof}

\begin{thm}\label{thm:3.5} 
Let $D$ be a weakly locally  finite division ring with center $F$, $K$ be a proper division subring of $D$ and suppose that $N$ is a normal subgroup of $D^*$. If $N$ is radical over $K$, then $N\subseteq F$.
\end{thm}

\begin{proof} 
Suppose that $N$ is not contained in the center $F$. If $N\setminus K=\emptyset$, then $N\subseteq K$. By [9, p. 433], either $K\subseteq F$ or $K=D$. Since $K\neq D$ by the assertion, it follows that $K\subseteq F$. Hence $N\subseteq F$, that contradicts to the assertion. Thus, we have $N\setminus K\neq\emptyset$.

Now, to complete the proof of our theorem we shall show that the elements of $N$ satisfy the requirements of Lemma \ref{lem:3.4}. Thus, suppose that  $a, b\in N$. We examine the following cases:

$1^0)$ {\em Case 1:}  $a\in K$.

{\em - Subcase 1.1:} $b\not\in K$. 

We shall prove that there exists some positive integer $n$ such that $a^nb=ba^n$. Thus, suppose that $a^nb\neq ba^n$ for any positive integer $n$. Then, $a+b\neq 0, a\neq \pm{1}$ and $b\neq \pm{1}$. So we have
$$x=(a+b)a(a+b)^{-1}, y=(b+1)a(b+1)^{-1}\in N.$$
Since $N$ is radical over $K$, we can find some positive integers $m_x$ and $m_y$ such that
$$x^{m_x}=(a+b)a^{m_x}(a+b)^{-1}, y^{m_y}=(b+1)a^{m_y}(b+1)^{-1}\in K.$$
Putting $m=m_xm_y$, we have
$$x^m=(a+b)a^m(a+b)^{-1}, y^m=(b+1)a^m(b+1)^{-1}\in K.$$
Direct calculations give the equalities
$$x^mb-y^mb+x^ma-y^m=x^m(a+b)-y^m(b+1)=(a+b)a^m-(b+1)a^m=a^m(a-1),$$
from that we get the following equality
$$(x^m-y^m)b=a^m(a-1)+y^m-x^ma.$$
If $(x^m-y^m)\neq 0$, then $b=(x^m-y^m)^{-1}[a(a^m-1)+y^m-x^ma]\in K$, that is a contradiction to the choice of $b$. Therefore $(x^m-y^m)= 0$ and consequently, $a^m(a-1)=y^m(a-1)$. Since $a\neq 1,a^m=y^m=(b+1)a^m(b+1)^{-1}$ and it follows that $a^mb=ba^m$, a contradiction.

{\em - Subcase 1.2:} $b\in K$. 

Consider an element $x\in N\setminus K$. Since $xb\not\in K$, by Subcase 1.1, there exist some positive integers $r, s$ such that 
$a^rxb=xba^r$ and $a^sx=xa^s.$
From these equalities it follows that
$a^{rs}=(xb)^{-1}a^{rs}(xb)=b^{-1}(x^{-1}a^{rs}x)b=b^{-1}a^{rs}b,$
and consequently, $a^{rs}b=ba^{rs}.$

$2^0)$ {\em Case 2:}  $a\not\in K$.

Since $N$ is radical over $K$, there exists some positive integer $m$ such that $a^m\in K$.  By Case 1, there exists some positive integer $n$ such that $a^{mn}b=ba^{mn}$.
\end{proof}

\section{Some skew linear groups} 

Let $D$ be a division ring with center $F$. In the following we identify $F^*$ with $F^*I:=\{\alpha I\vert~ \alpha\in F^*\}$, where $I$ denotes the identity matrix in $GL_n(D)$. 

In [7, Theorem 1], it was proved that if $D$ is centrally finite, then any finitely generated subnormal subgroup of $D^*$ is central. This result can be carried over for weakly locally finite division rings as the following.

\begin{thm}\label{thm:4.1} Let $D$ be a weakly locally finite division ring. Then, every finitely generated subnormal subgroup of $D^*$ is central.
\end{thm}
\begin{proof}  Since $N$ is finitely generated and $D$ is weakly locally finite, the division subring generated by $N$, namely $L$, is centrally finite.  By [7, Theorem 1], $N\subseteq Z(L)$. Consequently, $N$ is abelian. Now, by [9, 14.4.4, p. 440], $N\subseteq Z(D)$.
\end{proof}

The following theorem is a generalization of  Theorem 5 in \cite{akb3}.
\begin{thm}\label{thm:4.2}
Let $D$ be a weakly locally finite division ring with center $F$ and $N$ be a infinite  subnormal subgroup of $GL_n(D), n\geq 2$. If $N$ is finitely generated, then  $N\subseteq F$.
\end{thm}
\begin{proof}
Suppose that $N$ is non-central. Then, by  [8, Theorem 11], $SL_n(D)\subseteq N$. So, $N$ is normal in  $GL_n(D)$. Suppose that $N$ is generated by matrices $A_1, A_2, ..., A_k$ in $GL_n(D)$ and $T$ is the set of all coefficients of all $A_j$. 
Since $D$ is weakly locally finite, the division subring $L$ generated by $T$ is centrally finite.  It follows that $N$ is a normal finitely generated subgroup of $GL_n(L)$. By [1, Theorem 5], $N\subseteq Z(GL_n(L))$.
In particular,  $N$ is abelian and consequently, $SL_n(D)$ is abelian, a contradiction. 
\end{proof}

\begin{lem}\label{lem:4.3}
Let $D$ be a division ring and $n\geq 1$. Then, $Z(SL_n(D))$ is a torsion group if and only if $Z(D')$ is  a torsion group.
\end{lem}
\begin{proof}
The case $n=1$ is clear. So, we can assume that $n\geq 2$. Denote by $F$ the center of $D$. By [2, \S21, Theorem 1, p.140], 
$$Z(SL_n(D))=\big\{ dI \vert d\in F^* \text{ and } d^n\in D'\big\}.$$ 
If $Z(SL_n(D))$ is a torsion group, then, for any $d\in Z(D')=D'\cap F$, $dI\in Z(SL_n(D))$. It follows that $d$ is torsion.
Conversely, if $Z(D')$ is a torsion group, then, for any $A\in Z(SL_n(D))$, $A =dI$ for some $d\in F^*$ such that $d^n\in D'$. It follows that $d^n$ is torsion. Therefore, $A$ is torsion.
\end{proof}

\begin{thm}\label{thm:4.4}
 Let $D$ be a non-commutative algebraic, weakly locally finite division ring with center $F$ and  $N$ be a subgroup of $GL_n(D)$ containing $F^*, n\geq 1$.  Then $N$ is not finitely generated.
\end{thm}
\begin{proof} Recall  that if a division ring $D$ is weakly locally finite, then $Z(D')$ is a torsion group (see  Lemma \ref{lem:3.2}). Therefore, by Lemma \ref{lem:4.3}, $Z(SL_n(D))$ is a torsion group.

Suppose that there is  a finitely generated subgroup $N$ of $GL_n(D)$ containing $F^*$. Clearly $N/N'$ is a finitely generated abelian group, where $N'$ denotes the derived subgroup of $N$. Then, in virtue of [9, 5.5.8, p. 113], $F^*N'/N'$ is a finitely generated abelian group.\\

\noindent
{\em Case 1: $char(D)=0$.}

 Then, $F$ contains the field $\Q$ of rational numbers and it follows that $\Q^*I/(\Q^*I\cap N')\simeq \Q^*N'/N'$. Since   $F^*N'/N'$ is finitely generated abelian subgroup, $\Q^*N'/N'$ is finitely generated too, and consequently   $\Q^*I/(\Q^*I\cap N')$ is finitely generated. Consider an arbitrary $A\in \Q^*I\cap N'$. Then $A\in F^*I\cap SL_n(D)\subseteq Z(SL_n(D))$. 
Therefore $A$ is torsion.
Since $A\in \Q^*I$, we have $A=dI$ for some $d\in \Q^*$. It follows that $d=\pm{1}$. Thus, $\Q^*I\cap N'$ is finite. Since $\Q^*I/(\Q^*I\cap N')$ is finitely generated, $\Q^*I$ is finitely generated. Therefore $\Q^*$ is finitely generated, that  is impossible.\\

\noindent
{\em Case 2: $char(D)=p > 0$.}

 Denote by $\F_p$ the prime subfield of $F$, we shall prove that $F$ is algebraic over $\F_p$. In fact, suppose that $u\in F$ and $u$  is transcendental over $\F_p$. Put $K:=\F_p(u)$, then the group $K^*I/(K^*I\cap N')$ considered as a subgroup of $F^*N'/N'$ is finitely generated. Considering an arbitrary $A\in K^*I\cap N'$, we have $A=(f(u)/g(u))I$ for some $f(X), g(X)\in \F_p[X], ((f(X), g(X))=1$ and $g(u)\neq 0$. As mentioned above, we have $f(u)^s/g(u)^s=1$ for some positive integer $s$. Since $u$ is transcendental over $\F_p$,  $f(u)/g(u)\in \F_p$. Therefore, $K^*I\cap N'$ is finite and consequently, $K^*I$ is finitely generated. It follows that $K^*$ is finitely generated, hence $K$ is finite. Hence $F$ is algebraic over $\F_p$ and it follows that $D$ is algebraic over $\F_p$. Now, in virtue of Jacobson's Theorem  [6, (13.11), p. 219],  $D$ is commutative, a contradiction. 
\end{proof}
 
\begin{cor}\label{cor:4.5}
Let $D$ be an algebraic, weakly locally finite division ring. If the group $GL_n(D), n\geq 1$, is finitely generated, then $D$ is commutative.
\end{cor}

If $M$ is a maximal  finitely generated subgroup of $GL_n(D)$, then $GL_n(D)$ is finitely generated. So, the next result follows  immediately from  Corollary \ref{cor:4.5}.

\begin{cor}\label{cor:4.6}
Let $D$ be an algebraic, weakly locally finite division ring. If the group $GL_n(D), n\geq 1$,  has a maximal  finitely generated subgroup, then $D$ is commutative.
\end{cor}

By the same way as in the proof of Theorem \ref{thm:4.4}, we obtain the following corollary.

\begin{cor}\label{cor:4.7}
Let $D$ be a non-commutative  algebraic, weakly locally finite division ring with center $F$ and $S$ is a  subgroup of $GL_n(D)$. If $N=F^*S$, then $N/N'$ is not finitely generated.
\end{cor}

\begin{proof}
Suppose that $N/N'$ is finitely generated. Since $N'=S'$ and $ F^*I/(F^*I\cap S') \simeq F^*S'/S'$, $F^*I/(F^*I\cap S')$ is a finitely generated abelian group.  Now, by the same arguments as in the proof of Theorem \ref{thm:4.4}, we conclude that $D$ is commutative.
\end{proof}

\begin{cor}\label{cor:4.8}
Let $D$ be a non-commutative algebraic, weakly locally finite division ring. Then, $D^*$ is not finitely generated.
\end{cor}

\begin{proof} Take $N=S=GL_n(D)$ in Corollary \ref{cor:4.7} and have in mind that $[GL_n(D), GL_n(D)]=SL_n(D)$, we have $D^*\simeq GL_n(D)/SL_n(D)$ is not finitely generated.
\end{proof}

\noindent
{\bf Acknowledgments}

The authors express their sincere thanks to the editor and the referee for suggestions and important remarks.
The authors are also extremely grateful for the support given by Vietnam's National Foundation for Science and Technology Development (NAFOSTED)
under  grant number 101.01-2011.16


\begin{thebibliography}{99}

\bibitem{akb3} S. Akbari and M. Mahdavi-Hezavehi,  Normal subgroups of $GL_n(D)$ are not finitely generated, {\em Proc. of the Amer. Math. Soc.}, {\bf 128} (6) (2000), 1627-1632.

\bibitem{dra} P. Draxl, {\em Skew fields}, London Math. Soc. Lecture Note Series, No {\bf 81}, 1983.

\bibitem{hai-huynh}  Bui Xuan Hai and Le Khac Huynh,  On subgroups of the multiplicative group of a division ring,  {\em Vietnam Journal of Mathematics} {\bf 32:1} (2004), 21-24.

\bibitem{her} I. N. Herstein,  Multiplicative commutators in division rings,  {\em Israel Journal of Math.}, Vol. {\bf 31}, No. 2 (1978), 180-188.

\bibitem{Kha} V. K. Kharchenko, Simple, prime and semiprime rings, in: {\em Handbook of Algebra}, Vol. {\bf 1}, North-Holland, Amsterdam 1996, 761-812.

\bibitem{lam} T.Y. Lam,  {\em A First course in non-commutative rings}, GTM {\bf 131 }, Springer-Verlag, 1991.

\bibitem{mah1} M. Mahdavi-Hezavehi, M.G. Mahmudi, and S. Yasamin,  Finitely generated subnormal subgroups of $GL_n(D)$ are central, {\em Journal of Algebra} {\bf 255} (2000), 517-521.

\bibitem{mah2} M. Mahdavi-Hezavehi and S. Akbari,  Some special subgroups of $GL_n(D)$, {\em Algebra Colloq.}, {\bf 5:4} (1998), 361-370.

\bibitem{scott} W. R. Scott, {\em Group theory}, Dover Publication, INC, 1987.

\end{thebibliography}
\end{document}